\documentclass[a4paper,12pt]{amsart}

\usepackage{epsfig}
\usepackage{amsthm,amsfonts}
\usepackage{amssymb,graphicx,color}
\usepackage[all]{xy}

\newtheorem{theorem}{Theorem}[section]
\newtheorem{lemma}[theorem]{Lemma}

\newtheorem{proposition}[theorem]{Proposition}
\newtheorem{definition}[theorem]{Definition}
\newtheorem{example}[theorem]{Example}

\newcommand{\R}{\mathbb R}

\newcommand{\N}{\mathbb N}

\begin{document}

\title[Globally subanalytic  CMC surfaces in $\R^3$]{Globally subanalytic  CMC surfaces in $\R^3$}
\author{J. L. Barbosa}
\address{J. ~L. ~Barbosa, Rua Carolina Sucupira 723 ap 2002,
60140-120 Fortaleza-CE, Brazil} \email{joaolucasbarbosa@gmail.com}

\author{L. Birbrair}
\address{L. ~Birbrair and A. ~Fernandes, Departamento de Matem\'atica,Universidade Federal do Cear\'a Av.
Humberto Monte, s/n Campus do Pici - Bloco 914, 60455-760
Fortaleza-CE, Brazil} 
\email{birb@ufc.br}
 \email{alexandre.fernandes@ufc.br}

\author{M. do Carmo}
\address{M. ~do Carmo, Instituto Nacional de Matem\'atica Pura e Aplicada IMPA
Estrada Dona Castorina 110, 22460-320
Rio de Janeiro-RJ, Brazil} \email{manfredo@impa.br}

\author{A. Fernandes}

\thanks{The authors were partially supported by CNPq-Brazil}

\begin{abstract}
We prove that globally subanalytic nonsingular CMC surfaces of $\R^3$ are only planes, round spheres or right circular cylinders.   
\end{abstract}

\maketitle

\section{Introduction}
 
The question of describing the surfaces of constant mean curvature (CMC surfaces) is known in  Differential Geometry since the classical papers of Hopf \cite{H} and Alexandrov \cite{A}. In the preprint \cite{BC},  Barbosa and do Carmo showed that the connected nonsingular algebraic CMC surfaces in $\R^3$ are only the planes, spheres and cylinders. The present paper gives a generalization of this result to a larger setting where we consider CMC surfaces globally subanalytic in $\R^3$. In particular, one can find here a  proof for  semialgebraic connected nonsingular CMC surfaces. 
 
In the Section 2 we make a brief introduction to the theory of globally subanalytic sets for people who does not work in Real Algebraic Geometry. All this can be found in \cite{DM} or \cite{D}, and we will try to give explicit references where we can find most of results we use.

\section{Globally subanalytic sets}

In this section, we do a brief exposition of globally subanalytic sets following closely the paper \cite{D}.

\begin{definition}[Semianalytic sets] \rm{A subet $X\in\R^n$ is called \emph{semianalytic at} $x\in\R^n$ if there exists an open neighborhood $U$ of $x$ in $\R^n$ such that $U\cap X$ can be written as a finite union of sets of the form}
$ \{ x\in\R^n \ | \  p(x)=0, q_1(x)>0, \dots,q_k(x)>0\} $, \rm{where $p,q_1,\dots,q_k$ are analytic functions on $U$. A subset $X\subset\R^n$ is called \emph{semianalytic in} $\R^n$ if $X$ is semianalytic at each point $x\in\R^n$}.
\end{definition}

\begin{definition}[Subanalytic sets] \rm{ A subset $X\subset\R^n$  is called \emph{subanalytic at} $x\in\R^n$ if there exists an open neighborhood $U$ of $x$ in $\R^n$ and a bounded semianalytic subset $S\subset\R^n\times\R^m$, for some $m$, such that $U\cap X=\pi(S)$ where $\pi :\R^n\times\R^m\rightarrow\R^n$ is the orthogonal projection map. A subset $X\subset\R^n$ is called \emph{subanalytic in} $\R^n$ if $X$ is subanalytic at each poin of $\R^n$}
\end{definition}

\begin{example}\label{semialgebraic}\rm{
 The connected components of algebraic subsets of $\R^n$ are subanalytic subsets of $\R^n$. More generally, semialgebraic subsets of $\R^n$ are subanalytic subsets in $\R^n$( for the definition of semialgebraic sets, see \cite{C} page 24). }
\end{example} 

\begin{example}\rm{ Let $f :\R^n\rightarrow\R^m$ be an analytic function. Then, the zero set $X=\{x\in\R^n \ | \ f(x)=0\}$ is a subanalytic subset in $\R^n$. }
\end{example}

This example above shows us that some subanalytic subsets do not have finite topology. For instance, if $X\subset\R^2$ is the zero set of the function $\sin(x^2+y^2)$, then $X$ is a subanalytic set with infintely many connected components. 

\begin{definition}\rm{ A subset $X\subset\R^n$ is called \emph{globally subanalytic} if its image under the map, from $\R^n$ to $\R^n$,
$$(x_1,\dots,x_n)\mapsto (\frac{x_1}{\sqrt{1+x_1^2}},\dots,\frac{x_n}{\sqrt{1+x_n^2}})$$ is a subanalytic subset in $\R^n$.}
\end{definition}

\begin{example}\rm{  A map $f\colon\R^n\rightarrow\R^m$ is called a \emph{semialgebraic map} if its graph is a semialgebraic subset of $\R^n\times\R^m$. The Tarski-Seidenberg Theorem (see \cite{C} pages 26-27)  says that the image of the semialgebraic sets by the semialgebraic maps are also  semialgebraic sets. We use the Tarski-Seidenberg Theorem to provide an argument that the semialgebraic subsets of $\R^n$ are globally subanalytic sets. In fact, Let $X$ be a semialgebraic subset of $\R^n$. Since the map}

$$\begin{array}{cccc}
\varphi \ : & \! \R^n & \! \longrightarrow
& \! \R^n \\
& \! (x_1,\dots,x_n) & \! \longmapsto
& \! (\frac{x_1}{\sqrt{1+x_1^2}},\dots,\frac{x_n}{\sqrt{1+x_n^2}})
\end{array} 
$$

\noindent \rm{is a semialgebraic map, by Tarski-Seidenberg Theorem, $\varphi(X)$  is a semialgebraic subset of $\R^n$. By  Example \ref{semialgebraic}, $\varphi(X)$ is a subanalytic subset of $\R^n$. Finally, by definition, $X$ is globally subanalytic.}
\end{example}

According to the  next example, we see that the class of globally subanalytic sets is much larger than the class of semialgebraic ones.

\begin{example}\rm{ Let $p :\R^n\rightarrow\R^m$ be a rational map such that its image is included in a compact Euclidean ball $B$. Let $f : B\rightarrow\R$ be a function defined as a restriction to $B$ of an analytic function  $F: U\rightarrow\R$ where $U$ is an open neighborhood of $B$. Then, the graph and the level sets of the composite function $g=f\circ p$ are globally subanalytic subsets.}
\end{example}

\begin{definition}\rm{Let $X\subset\R^n$ be a globally subanalytic. A map $f:X\rightarrow\R^k$ is called a \emph{globally subanalytic map} if its graph is a globally subset of $\R^n\times\R^k$.}
\end{definition}

\begin{example}\rm{ The function $f$ defined by $(x,y)\mapsto \displaystyle \cos(\frac{1}{1+x^2+y^2})$ is a globally subanalytic function. In particular, the graph of $f$ is a surface of class $C^{\omega}$, properly embedded in $\R^3$, which  is a globally subanalytic set. }
\end{example} 

\begin{example}\rm{ Let $X\subset\R^n$ be a globally subanalytic subset. The distance function to $X$, from $\R^n$ to $[0,\infty)$, defined by, $p\mapsto \rm{dist}(p,X)$, is a globally subanalytic function.}
\end{example} 

In order to show the strength of the theory of globally subanalytic sets, let us list some important  properties.

\begin{enumerate}

\item {\bf Triangulability} (see \cite{L}): {\it Let $X\subset\R^n$ be a compact subanalytic subset. Let  $B_i$, $i = 1, . . . , k,$ be suabanalytic subsets of $X$. Then there exist a finite simplicial
complex $K$  and a subanalytic homeomorphism $\Phi : |K| \rightarrow X$ such that each $B_i$ is a union of images by $\Phi$ of open simplices of $K.$}

\item {\bf Piecewise Monotonicity} (see \cite{D} page 192): {\it  if $f\colon(a,b)\rightarrow\R$ is a globally subanalytic function, then there exists a partition $a=a_0<a_1<\cdots<a_m=b$ such that on each interval $(a_i,a_{i+1})$ the function $f$ is either continuous and strictly monotone or constant.}

\item {\bf Polynomial Growth} (see \cite{D} page 192) : Given a globally subanalytic function $ f: (x_0, \infty )\rightarrow \R$, there are $d \in \N$ and $a > 0$ such that $ |f(t)|< t^d$ for $t > a.$


\end{enumerate}

In order to finish this section, as a consequence of the Polynomial Growth Property,  let us point out that exponential functions $\R\rightarrow\R$ are not globally subanalytic functions, as one can see in the next result. 

\begin{proposition}\label{proposition1} Let $f\colon (0,\infty)\rightarrow\R$ be a globally subanalytic function such that there exist positive real numbers $c$ and $\lambda$;  $0\leq f(R_n)\leq c e^{-\lambda R_n},$ where $\{R_n\}$ is a sequence of positive real numbers such that $R_n\to \infty$. Then, $f(R)=0$ for all $R$ large enough. 
\end{proposition}

\begin{proof}If the proposition is false for the function $f$,  by the Piecewise Monotonicity Property, there exists a positive real number $x_0$ such that $f(R)>0$ for all $R>x_0$. Let $g\colon[x_0,\infty)\rightarrow(0,\infty)$ be defined by $g(t)=1/f(t)$. Since $g$ is a globally subanalytic function, by the Polynomial Growth Property, there are $d \in \N$ and $a > 0$ such that $ 0<g(t)< t^d$ for $t > a$, what, jointly with the hypothesis, implies that $e^{\lambda R_n}\leq cR_n^d, \ \forall n$. Since  $R_n\to \infty$, the last inequality is an obvious absurd.   
\end{proof}

\section{Main theorem and its proof}

Let $X\subset\R^3$ be a surface non-singular, analytic and closed  in $\R^3$. Let us suppose that $X\subset\R^3$ is globally subanalytic. First, we are going to show that $X$ has finite topology.  In fact, we prove a  more general result, that is, all globally subanalytic sets have finite topology. 

\begin{lemma}  X has finite topology.
\end{lemma}

\begin{proof} If $X$ is compact, we have nothing to do. Then let us suppose that $X$ is not compact. Let $S\subset\R^4$ be a round sphere. Let $p\in S$ and $\phi : S\setminus \{p\}\rightarrow\R^3$ be the stereographic projection associated to the point $p\in S$. Since $\phi$ is a rational homeomorphism, $\phi^{-1}(X)$ is a subanalytic subset of $S$ and its compactification, $\tilde{X}=\phi^{-1}(X)\cup \{p\}$, is a compact subanalytic subset of $S$. By the Triangulability Property (see Section 2), we can take a finite triangulation $\tau$ of $\tilde{X}$ where $p$ is a vertex of $\tau$. Thus, removing the point $p$, we see that $X$ is homeomorphic to a compact polyhedron minus a finte subset of points.
\end{proof}

\begin{theorem} If $X$ has constant mean curvature, then $X$ is a plane, a sphere or a right circular cylinder.
\end{theorem}

\begin{proof} 
 Let us suppose the mean curvature $H$ of $X$ is a nonzero constant. The case $H=0$ will be analyzed later on. If $X$ is compact, by Aleksandrov's Theorem, $X$ is a sphere.
Thus, the problem is reduced to the case where $X$ is not compact. Let $\Sigma$ be an end of $X$. In other words, let $\Sigma$ be a connected component of the set
$$\{x\in X \ | \ dist(x,0)\geq R_0\}$$ for a sufficiently big fixed $R_0>0$. It follows from the paper \cite{KKS} that there exists a Delaunay surface $D$ which exponentially approximates $\Sigma$, that is, there are $C>0$ and $\lambda >0$ such that
\begin{equation}\label{approximation-nonminimal}
{\rm Hausdorff \ Distance} \ (\Sigma_R,D_R)\leq Ce^{-\lambda R}, \ \forall \ R>0,
\end{equation}  where  
$$\Sigma_R=\{x\in \Sigma \ | \ dist(x,0)\leq R\} \ \mbox{and} \ D_R=\{x\in D \ | \ dist(x,0)\leq R\}.$$ 

{\bf Claim 1.} {\it  $D$ is a right circular cylinder.}

\begin{proof}[Proof of Claim 1] In fact, if $D$ is not a right circular cylinder, by classification of Delaunay surfaces, $D$ is an onduloid. Then, there exists a half-line $L$ satisfying

\begin{itemize}
\item[a)] $L\cap D=\{x_n\}$ is a sequence of points such that $|x_n|\to \infty$;
\item[b)] $\exists$ $\{y_n\}\subset L$ sequence of points and a real number $k>0$ such that $\rm{dist}(y_n,D)>k$ for all $n$, and $|y_n|\to \infty$.
\end{itemize}

For each $R$ large enough, let $q_R$ be the only point in $L$ such that $|q_R|=R$,. Then, the real function $h(R):= \rm{dist}(q_R,\Sigma_R)$ is a globally subanalytic function (see the argument below for the function $f$ in (\ref{function f})). Since $x_n\in D$, it comes from Inequality (\ref{approximation-nonminimal}) that 
$$0\leq h(|x_n|)\leq Ce^{-\lambda|x_n|}, \ \forall \ n.$$ Thus, using Proposition \ref{proposition1}, we get that $h(R)=0$ for all $R$ large enough. Since $|y_n|\to \infty$, we have that $\exists n_0$ such that $y_n\in\Sigma$ for all $n>n_0$. But, this is an absurd, because $\rm{dist}(y_n,D)>k$ for all $n$ and $\rm{Hausdorff \ Distance} (\Sigma_{|y_n|}, D_{|y_n|})\to 0$.
\end{proof}

Once we have proved that  $D$ is a right circular cylinder, we are going to show that $X$ is equal to $D$. In fact, let us denote by $f$ the following function
\begin{equation}\label{function f}
f(R)={\rm Hausdorff \ Distance} \ (\Sigma_R,D_R).
\end{equation}
 Since the graph of $f$ can be defined as the set of points in $\R^2$ that satisfies a first-order formula in the language of the structure of the globally subanalytic sets (see Theorem 1.13 of \cite{C2}),  the function $f$ is globally subanalytic . Since 
$$0\leq f(R)\leq e^{-\lambda R}, \ \forall \ R>0,$$ we have that $f(R)=0$ for $R$ large enough. Thus, by analicity, $D$ is equal to $X$. 

Now, let $H=0$. Let $\Sigma$ be an end of $X$. Without loss of generality we may suppose  the $z$-axis in $\R^3$ is the asymptotic direction of $\Sigma$ (see \cite{O}). Thus, outside a big Euclidean ball in $\R^3$, $\Sigma$ is the graph of a function $z$ defined in $$\{(x,y)\in\R^2 \ : x^2+y^2\geq R^2\}.$$ By the result  of R. Schoen in \cite{S},  there exist constant real numbers $a$, $b$, $c$ and $d$  such that
\begin{equation}\label{approximation-minima}
 z(x,y) = a\log(t) + b + t^{-2}(cx+dy)+O(t^{-2})
\end{equation}
where $t=\sqrt{x^2+y^2}$. 

\bigskip

{\bf Claim 2.} {\it The constant $a$ is equal to zero}.

\begin{proof}[Proof of Claim 2] Let us suppose that $a\not =0$. Without lost generality, we can suppose that $a>0$. Since $z(x,y)$ is a globally subanalytic function, the function $\zeta : [R,\infty)\rightarrow\R$ defined as the restriciton of $z(x,y)-b$ to the line $x=y$, that is,
$\zeta(x)=z(x,x)-b$ is a globallay subanalytic function. By (\ref{approximation-minima}), there exist $\varepsilon>0$ and $R_0>R$ such that $$0\leq \zeta(x)\leq \varepsilon a \log(x) \ \forall \ x\geq R_0.$$  By the Piecewise Monotonicity Property, we can choose $R_0>0$ large enough such that the restriction of $\zeta$ to the interval $[R_0,\infty)$ is a continuous and strictly monotone function. Let $g : [S_0,\infty)\rightarrow [R_0,\infty)$ be the inverse function of $\zeta : [R_0,\infty)\rightarrow\R$, in particular $g$ is a globally subanalytic function. By the last inequality above, $$0\leq g(s)\leq e^{-\frac{s}{\varepsilon a}} \ \forall \ s\geq S_0$$ hence $g(s)=0$ for all $s$ large enough. What is an absurd.
\end{proof}

Once we have proved that $a=0$, we get that  X is asymptotic to the plane $z=b$. Thus,  each end of $X$ is asymptotic to some plane.  Since $X$ is a properly embedded surfaces in $\R^3$ with  finitely many ends, the asymptotic planes to $X$ must be parallel planes.  It follows that there are two planes and $X$ is contained in the region bounded by them. It follows by the halfspace theorem for minimal surface \cite{HM} that $X$ must be a plane. 

\end{proof}

Finally, we would like to observe that the theorem above holds true  if we assume that  the surface $X$ is definable in a polynomial bounded structure on $\R$, instead  of the assumption  that the surface  is a  globally subanalytic set.

\end{document}